\documentclass{amsart}

\usepackage[T1]{fontenc}
\usepackage[latin1]{inputenc}
\usepackage{amsmath,amsthm,amstext,amsfonts,amssymb,amscd,amsxtra}
\usepackage{bm}
\usepackage{enumerate}
\usepackage{eucal}
\usepackage{mathrsfs}
\usepackage{stmaryrd}
\usepackage[all]{xy}

\allowdisplaybreaks

\theoremstyle{plain}
\newtheorem{theorem}{Theorem}[section]

\newtheorem*{fact}{Fact}
\newtheorem*{theoremA}{Theorem~A}
\newtheorem*{corollaryB}{Corollary~B}
\newtheorem*{corollaryC}{Corollary~C}
\newtheorem{lemma}[theorem]{Lemma}

\newtheorem{proposition-definition}{Proposition-Definition}
\newtheorem{corollary}[theorem]{Corollary}
\newtheorem{claim}[theorem]{Claim}

\theoremstyle{definition}

\theoremstyle{remark}
\newtheorem{remark}[theorem]{Remark}

\newcommand{\ZZ}{\mathbb{Z}}
\newcommand{\QQ}{\mathbb{Q}}
\newcommand{\RR}{\mathbb{R}}
\newcommand{\CC}{\mathbb{C}}

\newcommand{\Affsp}{\mathbb{A}}
\newcommand{\PP}{\mathbb{P}}

\makeatletter
 
 \@addtoreset{equation}{section}
\makeatother

\makeatletter
\newcommand{\esssup}{\mathop{\operator@font ess.sup}\displaylimits}
\newcommand{\essinf}{\mathop{\operator@font ess.inf}\displaylimits}
\newcommand{\trdeg}{\mathop{\operator@font tr.deg}\displaylimits}
\makeatother

\makeatletter 
\let\@@pmod\pmod 
\DeclareRobustCommand{\pmod}{\@ifstar\@pmods\@@pmod} 
\def\@pmods#1{\mkern4mu({\operator@font mod}\mkern 6mu#1)} 
\makeatother

\renewcommand{\div}{\mathop{\mathrm{div}}\nolimits}
\newcommand{\Image}{\mathop{\mathrm{Image}}\nolimits}
\renewcommand{\leq}{\leqslant}
\renewcommand{\geq}{\geqslant}

\DeclareMathOperator{\Hz}{H^0}

\DeclareMathOperator{\Spec}{Spec}

\DeclareMathOperator{\Card}{Card}

\DeclareMathOperator{\Bs}{Bs}
\DeclareMathOperator{\SBs}{SBs}

\DeclareMathOperator{\Sing}{Sing}

\DeclareMathOperator{\rk}{rk}

\newcommand{\id}{{\rm id}}
\newcommand{\pr}{{\rm pr}}
\newcommand{\chr}{{\rm char}}

\newcommand{\sbullet}{{\scriptscriptstyle\bullet}}

\title{A Bertini-type theorem for free arithmetic linear series}
\author{Hideaki Ikoma}
\date{November 23, 2013. Version 2.0.}
\thanks{This research is supported by Research Fellow of Japan Society for the Promotion of Science.}
\address{Graduate School of Mathematical Sciences, The University of Tokyo, Tokyo, 153-8914, Japan}
\email{ikoma@ms.u-tokyo.ac.jp}
\begin{document}

\begin{abstract}
In this paper, we prove a version of the arithmetic Bertini theorem asserting that there exists a strictly small and generically smooth section of a given arithmetically free graded arithmetic linear series.
\end{abstract}

\maketitle
\tableofcontents

\setcounter{section}{-1}
\section{Introduction}

When we generalize results on arithmetic surfaces to those on higher-dimensional arithmetic varieties, it is sometimes very useful to cut the base scheme by a ``good'' global section $s$ of a given Hermitian line bundle and proceed to induction on dimension.
To do this, we have in the context of Arakelov geometry the following result.

\begin{fact}[\text{\cite[Theorems~4.2 and 5.3]{Moriwaki95}}]
Let $\overline{A}$ be a $C^{\infty}$-Hermitian line bundle on a generically smooth projective arithmetic variety $X$, and let $x_1,\dots,x_q$ be points (not necessarily closed) on $X$.
Suppose that (i) $A$ is ample, (ii) $c_1(\overline{A})$ is positive definite, and (iii) $\Hz(X,mA)$ has a $\ZZ$-basis consisting of sections with supremum norms less than $1$ for every $m\gg 1$.
Then there exist a sufficiently large integer $m\geq 1$ and a nonzero section $s\in\Hz(X,mA)$ such that
\begin{enumerate}
\item[\textup{(1)}] $\div(s)_{\QQ}$ is smooth over $\QQ$,
\item[\textup{(2)}] $s(x_i)\neq 0$ for every $i$, and
\item[\textup{(3)}] $\|s\|_{\sup}<1$.
\end{enumerate}
\end{fact}

For example, this technique plays essential roles in the proofs of the arithmetic Bogomolov-Gieseker inequality on high-dimensional arithmetic varieties (see \cite{Moriwaki95}), of the arithmetic Hodge index theorem in codimension $1$ (see \cite{Moriwaki96}, \cite{Yuan_Zhang13}), of the arithmetic Siu inequality of Yuan (see \cite{Yuan07}), and so on.
A purpose of this paper is to give a simple elementary proof of the above fact and to strengthen it to the case of arithmetically free graded arithmetic linear series.

Let $K$ be a number field.
Let $X$ be a projective arithmetic variety that is geometrically irreducible over $\Spec(O_K)$, and let $L$ be an effective line bundle on $X$.
A \emph{graded linear series} belonging to $L$ is a subgraded $O_K$-algebra
\[
 R_{\sbullet}:=\bigoplus_{m\geq 0}R_m\subseteq \bigoplus_{m\geq 0}\Hz(X,mL).
\]
We consider norms $\|\cdot\|_m$ on $R_m\otimes_{\ZZ}\RR$, and assume that the family of norms $\|\cdot\|_{\sbullet}:=(\|\cdot\|_m)_{m\geq 0}$ is \emph{multiplicative}, that is,
\[
 \|s\otimes t\|_{m+n}\leq\|s\|_m\|t\|_n
\]
holds for every $s\in R_m$ and $t\in R_n$.

\begin{theoremA}
Let $X$ be a generically smooth projective arithmetic variety, and let $A$ be an effective line bundle on $X$.
We consider a graded linear series
\[
 R_{\sbullet}:=\bigoplus_{m\geq 0}R_m
\]
belonging to $A$ and a multiplicative norm $\|\cdot\|_{\sbullet}$ on $R_{\sbullet}\otimes_{\ZZ}\RR$.
Suppose the following conditions.
\begin{itemize}
\item $R_1$ is base point free,
\item $R_{\sbullet}\otimes_{\ZZ}\QQ$ is generated by $R_1$ over $\QQ$, and
\item $\bigcap_{m\geq 1}\{x\in X_{\QQ}\,|\,\text{$t(x)=0$ for every $t\in R_m$ with $\|t\|_m<1$}\}=\emptyset$.
\end{itemize}
Let $Y^1,\,\dots,\,Y^p$ be smooth closed subvarieties of the complex manifold $X(\CC)$, and let $x_1,\dots,x_q$ be points (not necessarily closed) on $X$.
Then, for every sufficiently large integer $m\gg 1$, there exists a nonzero section $s\in R_m$ such that
\begin{enumerate}
\item[\textup{(1)}] $\div(s|_{Y^1}),\,\dots,\,\div(s|_{Y^p})$ are all smooth,
\item[\textup{(2)}] $s(x_i)\neq 0$ for every $i$, and
\item[\textup{(3)}] $\|s\|_m<1$.
\end{enumerate}
\end{theoremA}

Let $\overline{L}$ be a continuous Hermitian line bundle on $X$, and let $\|\cdot\|_{\sup}^{(m)}$ be the supremum norm on $\Hz(X,mL)\otimes_{\ZZ}\RR$.
We define a $\ZZ$-submodule of $\Hz(X,mL)$ by
\[
 \mathrm{F}^{0+}(X,m\overline{L}):=\left\langle s\in\Hz(X,mL)\,\left|\,\|s\|_{\sup}^{(m)}<1\right.\right\rangle_{\ZZ}.
\]
Then $\bigoplus_{m\geq 0}\mathrm{F}^{0+}(X,m\overline{L})$ is a graded linear series belonging to $L$.
We denote the stable base locus of $\bigoplus_{m\geq 0}\mathrm{F}^{0+}(X,m\overline{L})$ by $\SBs^{0+}(\overline{L})$.

\begin{corollaryB}
Let $X$ be a generically smooth projective arithmetic variety, and let $\overline{A}$ be a continuous Hermitian line bundle on $X$.
Suppose that $\SBs(A)=\emptyset$ and $\SBs^{0+}(\overline{A})\cap X_{\QQ}=\emptyset$.
Let $Y^1,\,\dots,\,Y^p$ be smooth closed subvarieties of the complex manifold $X(\CC)$, and let $x_1,\dots,x_q$ be points (not necessarily closed) on $X$.
Then there exist a sufficiently large integer $m\geq 1$ and a nonzero section $s\in\Hz(X,mA)$ such that
\begin{enumerate}
\item[\textup{(1)}] $\div(s|_{Y^1}),\,\dots,\,\div(s|_{Y^p})$ are all smooth,
\item[\textup{(2)}] $s(x_i)\neq 0$ for every $i$, and
\item[\textup{(3)}] $\|s\|_{\sup}^{(m)}<1$.
\end{enumerate}
\end{corollaryB}

\begin{corollaryC}
Let $X$ be a generically smooth normal projective arithmetic variety, let $\overline{L}:=(L,|\cdot|_{\overline{L}})$ be a continuous Hermitian line bundle on $X$, and let $x_1,\dots,x_q$ be points (not necessarily closed) on $X\setminus\SBs^{0+}(\overline{L})$.
If $\SBs^{0+}(\overline{L})\subsetneq X$, then there exist a sufficiently large integer $m\geq 1$ and a nonzero section $s\in\Hz(X,mL)$ such that
\begin{enumerate}
\item[\textup{(1)}] $\div(s)_{\QQ}$ is smooth off $\SBs^{0+}(\overline{L})$,
\item[\textup{(2)}] $s(x_i)\neq 0$ for every $i$, and
\item[\textup{(3)}] $\|s\|_{\sup}^{(m)}<1$.
\end{enumerate}
\end{corollaryC}

\subsection*{Notation and conventions}

Let $k$ denote a field, and let $\PP^n:=\PP(k^{n+1})$ denote the projective space of one-dimensional quotients of $k^{n+1}$.
Let $\pr_2:\PP^n\times_k\PP^m\to\PP^m$ denote the second projection.
We denote the natural coordinate variables of $\PP^n$ (resp., of $\PP^m$) by $X_0,\dots,X_n$ (resp., by $Y_0,\dots,Y_m$) or simply by $X_{\sbullet}$ (resp., by $Y_{\sbullet}$).

Let $Y$ be a smooth variety over $k$.
The \emph{singular locus} of a morphism $\varphi:X\to Y$ over $k$ is a Zariski-closed subset of $X$ defined as
\[
 \Sing(\varphi):=\{x\in X\,|\,\text{$\varphi$ is not smooth at $x$}\}.
\]

A \emph{projective arithmetic variety} $X$ is a reduced irreducible scheme that is projective and flat over $\Spec(\ZZ)$.
We say that $X$ is \emph{generically smooth} if $X_{\QQ}:=X\times_{\Spec(\ZZ)}\Spec(\QQ)\to\Spec(\QQ)$ is smooth.

\section{Bertini's theorem with degree estimate}\label{sec:Bertini}

In this section, we consider the geometric case.
Let $X\subseteq\PP^n$ be a projective variety over an algebraically closed field $k$ that is defined by a homogeneous prime ideal $I_X\subseteq k[X_0,\dots,X_n]$, let $\mathcal{O}_X(1)$ be the hyperplane line bundle on $X$, and let
\[
 \deg X:=\deg(c_1(\mathcal{O}_X(1))^{\cdot\dim X})
\]
be the degree of $X$ in $\PP^n$.
Let $k[X]:=k[X_0,\dots,X_n]/I_X$ be the homogeneous coordinate ring of $X$, and let $k[X]_l$ be the homogeneous part of $k[X]$ of degree $l$.
There exists a polynomial function $\varphi_X(l)$ such that $\deg\varphi_X=\dim X$, all coefficients are nonnegative, and
\begin{equation}\label{eqn:varphil}
 \dim_kk[X]_l\leq \varphi_X(l)
\end{equation}
for all $l\geq 0$.
Let $Z\subseteq X\times_k\PP^m$ be a Zariski-closed subset defined by a system of polynomial equations:
\[
 u_1(X_{\sbullet};Y_{\sbullet})=0\pmod*{I_X},\,\dots,\, u_h(X_{\sbullet};Y_{\sbullet})=0\pmod*{I_X},
\]
where $u_i\in k[X_0,\dots,X_n;Y_0,\dots,Y_m]$ has homogeneous degree $\deg_{X_{\sbullet}}\! u_i$ (resp., $\deg_{Y_{\sbullet}}\! u_i$) in the set of variables $X_{\sbullet}$ (resp., $Y_{\sbullet}$).
We recall the following fact from the elimination theory.

\begin{lemma}\label{lem:degree}
Let $p:=\max_i\{\deg_{X_{\sbullet}}\! u_i\}$ and let $q:=\max_i\{\deg_{Y_{\sbullet}}\! u_i\}$.
If the set-theoretic image $\pr_2(Z)$ does not coincide with $\PP^m$, then $\pr_2(Z)$ is contained in a hypersurface of $\PP^m$ defined by a single homogeneous polynomial of degree less than or equal to
\[
 \varphi_X(\deg X\cdot p^{\dim X+1})\cdot q.
\]
\end{lemma}

\begin{proof}
First, we can take a geometric point $y_{0,\sbullet}=(y_{0,0}:\dots:y_{0,m})\in\PP^m\setminus \pr_2(Z)$.
By an effective Nullstellensatz \cite[Corollary 1.4]{Jelonek05}, there exists a positive integer $\ell\leq\deg X\cdot p^{\dim X+1}$ such that
\[
 (X_0,\dots,X_n)^{\ell}\subseteq (u_1(X_{\sbullet};y_{0,\sbullet}),\dots,u_h(X_{\sbullet};y_{0,\sbullet}))\pmod{I_X}.
\]
Next, we consider the $k$-linear maps
\[
 \begin{array}{cccc} T(y_{\sbullet}): & k[X]_{\ell-\deg_{X_{\sbullet}}\! u_1}\oplus\dots\oplus k[X]_{\ell-\deg_{X_{\sbullet}}\! u_h} & \to & k[X]_{\ell}\\ & (f_1(X_{\sbullet}),\dots,f_h(X_{\sbullet})) & \mapsto & \sum_iu_i(X_{\sbullet};y_{\sbullet})f_i(X_{\sbullet})\end{array}
\]
defined for $y_{\sbullet}=(y_0:\dots:y_m)\in\PP^m$.
By fixing basis for the above $k$-vector spaces, we can represent $T(y_{\sbullet})$ by a matrix whose entries are homogeneous polynomials of $y_{\sbullet}$ of degree less than or equal to $q$.
By the choice of $\ell$, we can see that there exists a certain $\dim_kk[X]_{\ell}\times\dim_kk[X]_{\ell}$-minor of the representation matrix of $T(y_{\sbullet})$ whose determinant is nonzero (see \cite[Theorem (2.23)]{Mumford95}).
Then the image $\pr_2(Z)$ is contained in the hypersurface defined by the nonzero determinant, which is homogeneous of degree less than or equal to $(\dim_kk[X]_{\ell})\cdot q$.
Since
\[
 \dim_kk[X]_{\ell}\leq\varphi_X(\ell)\leq\varphi_X(\deg X\cdot p^{\dim X+1}),
\]
we have the result.
\end{proof}

\begin{remark}
For example, we consider the case where $X=\PP^n$.
Then $\dim_kk[X]_l=\binom{l+n}{n}\leq (l+n)^n/n!$.
Thus the bound in the above lemma becomes less than or equal to $(p^{n+1}+n)^nq/n!$.
Moreover, by applying the theory of resultants (see \cite[page 35]{Mumford95}) to $\pr_2:\PP^n\times_k\Affsp^m\to\Affsp^m$, one can obtain a weaker bound less than or equal to $(2p)^{2^n-1}q+1$ in the above lemma (where the added $1$ is for the hyperplane at infinity).
\end{remark}

Let $A$ be an effective line bundle on $X$, and let $R_{\sbullet}$ be a subgraded ring of $\bigoplus_{m\geq 0}\Hz(X,mA)$ with Kodaira-Iitaka dimension $\kappa(R_{\sbullet}):=\trdeg_kR_{\sbullet}-1$.
Suppose that $R_1$ is base point free.
Let $\phi_m:X\to\PP(R_m)$ be a $k$-morphism associated to $R_m$, and set
\begin{equation}
 N_m:=\dim_kR_m-1
\end{equation}
for $m\geq 1$.
We recall that the rational function field $k(X)$ of $X$ is given by
\[
 k(X)=\left\{\frac{u\pmod*{I_X}}{v\pmod*{I_X}}\,\left|\,\begin{array}{l}\text{$u,v\in k[X_0,\dots,X_n]$ are homogeneous}\\ \text{of the same degree and $v\notin I_X$}\end{array}\right.\right\}.
\]
Given a nonzero section $e\in R_1$, we define the \emph{degree} of a nonzero section $s\in\Hz(X,mA)$ for $m\geq 1$ with respect to $e$ by
\[
 \deg_{X_{\sbullet},e}\! s:=\min\left\{\deg_{X_{\sbullet}}\! u=\deg_{X_{\sbullet}}\! v\,\left|\,\begin{array}{l}\div s=\left(u/v\pmod*{I_X}\right)+m\div e,\\ u/v\pmod*{I_X}\in k(X)^{\times}\end{array}\right.\right\}.
\]
(Compare the definition with Jelonek's in \cite[\S 2]{Jelonek05}.)
Then, for any other nonzero section $s'\in\Hz(X,m'A)$, we have
\[
 \deg_{X_{\sbullet},e}\!(s\otimes s')\leq\deg_{X_{\sbullet},e}\! s+\deg_{X_{\sbullet},e}\! s'.
\]

\begin{theorem}\label{thm:degBertini}
Let $X\subseteq\PP^n$ be a smooth projective variety over $k$, and let $A$ be a line bundle on $X$.
Let $R_{\sbullet}$ be a graded linear series belonging to $A$ with Kodaira-Iitaka dimension $\kappa(R_{\sbullet})$.
Suppose that the following three conditions are satisfied.
\begin{itemize}
\item $R_1$ is base point free.
\item $R_{\sbullet}$ is generated by $R_1$.
\item (i) $\chr(k)=0$ or (ii) $\chr(k)\neq 0$ and $\phi_m:X\to\PP(R_m)$ is unramified for every $m\geq 1$.
\end{itemize}
Then one can find a polynomial function $P(m)$ and hypersurfaces $Z_m\subseteq\PP(R_m^{\vee})$ for $m=1,2,\dots$ having the following two properties.
\begin{enumerate}
\item[\textup{(1)}] $\deg P\leq\dim X(\dim X+1)(\kappa(R_{\sbullet})+1)$.
\item[\textup{(2)}] For every $m\geq 1$, the hypersurface $Z_m\subseteq\PP(R_m^{\vee})$ contains the set
\[
 \{H\in\PP(R_m^{\vee})\,|\,\phi_m(X)\subseteq H\text{ or }\phi_m^{-1}(H)\text{ is not smooth}\}
\]
and the homogeneous degree of $Z_m$ in $\PP(R_m^{\vee})$ is less than or equal to $P(m)$.
\end{enumerate}
\end{theorem}

\begin{remark}
Throughout this paper, we assume that the empty set $\emptyset$ is smooth, so that, if $H\notin Z_m$, then $\phi_m^{-1}(H)$ is empty or smooth of pure dimension $\dim X-1$.
\end{remark}

\begin{proof}
Let $I_X\subseteq k[X_0,\dots,X_n]$ denote the homogeneous prime ideal defining $X$.
We consider the universal hyperplane section
\begin{equation}
 W_m:=\{(x,H)\in X\times_k\PP(R_m^{\vee})\,|\,\phi_m(x)\in H\}
\end{equation}
endowed with the reduced induced scheme structure, and consider the restriction of the second projection $\pr_2:X\times_k\PP(R_m^{\vee})\to\PP(R_m^{\vee})$ to $W_m$, which we denote by
\begin{equation}
 \pi_m:W_m\to\PP(R_m^{\vee}).
\end{equation}
Note that $W_m$ is the inverse image of the canonical bilinear hypersurface in $\PP(R_m)\times_k\PP(R_m^{\vee})$ via $\phi_m\times\id:X\times_k\PP(R_m^{\vee})\to\PP(R_m)\times_k\PP(R_m^{\vee})$.
Since the restriction of the first projection to $W_m$, $W_m\to X$, is surjective with fiber a projective space of dimension $N_m-1$, $W_m$ is irreducible.
The set-theoretic image of the singular locus of $\pi_m$ is given by
\[
 \pi_m(\Sing(\pi_m))=\{H\in\PP(R_m^{\vee})\,|\,\phi_m(X)\subseteq H\text{ or }\phi_m^{-1}(H)\text{ is not smooth}\}.
\]

We fix a basis $e_0,\dots,e_{N_1}$ for $R_1$.
From now on, we explain the method to construct an equation $w_0$ that vanishes along $W_m$ from the section $e_0$.
First, we set
\begin{equation}
 D_{1,e_0}:=\max_{1\leq i\leq N_1}\{\deg_{X_{\sbullet},e_0}\!e_i\},
\end{equation}
and take rational functions $u_1^{(1)}/v_1^{(1)},\dots,u_{N_1}^{(1)}/v_{N_1}^{(1)}\in k(X_0,\dots,X_n)^{\times}$ such that
\[
 \div e_i=\left(\frac{u_i^{(1)}}{v_i^{(1)}}\pmod*{I_X}\right)+\div e_0\quad\text{and}\quad\deg_{X_{\sbullet}}\!u_i^{(1)}=\deg_{X_{\sbullet}}\! v_i^{(1)}\leq D_{1,e_0}
\]
for $i=1,\dots,N_1$.
Next, for $m\geq 2$, we can choose sections $e_1^{(m)},\dots e_{N_m}^{(m)}\in R_m$ such that
\[
 e_i^{(m)}\in\left\{\left.e_0^{\otimes\alpha_0}\otimes\dots\otimes e_{N_1}^{\otimes\alpha_{N_1}}\,\right|\,\alpha_0+\dots +\alpha_{N_1}=m\right\}
\]
and $e_0^{\otimes m},e_1^{(m)},\dots,e_{N_m}^{(m)}$ form a basis for $R_m$.
By identifying $\PP(R_m^{\vee})$ with $\PP^{N_m}$ via the dual basis of $e_0^{\otimes m},e_1^{(m)},\dots,e_{N_m}^{(m)}$, we can write $\phi_m:X\to\PP(R_m^{\vee})$ as
\[
 \phi_m:X_{e_0}\to\PP^{N_m},\quad x\mapsto\left(1:\frac{u_1^{(m)}(x)}{v_1^{(m)}(x)}:\dots :\frac{u_{N_m}^{(m)}(x)}{v_{N_m}^{(m)}(x)}\right)
\]
over $X_{e_0}:=\{x\in X\,|\,e_0(x)\neq 0\}$, where $u_i^{(m)}/v_i^{(m)}\in k(X_0,\dots,X_n)^{\times}$ satisfies
\[
 \div e_i^{(m)}=\left(\frac{u_i^{(m)}}{v_i^{(m)}}\pmod*{I_X}\right)+m\div e_0\quad\text{and}\quad\deg_{X_{\sbullet}}\!u_i^{(m)}=\deg_{X_{\sbullet}}\! v_i^{(m)}\leq D_{1,e_0}m.
\]
We set
\begin{equation}
 w_0:=v_1^{(m)}\cdots v_{N_m}^{(m)}Y_0+u_1^{(m)}v_2^{(m)}\cdots v_{N_m}^{(m)}Y_1+\dots +v_1^{(m)}\cdots v_{N_m-1}^{(m)}u_{N_m}^{(m)}Y_{N_m},
\end{equation}
which is homogeneous in $X_{\sbullet}$ (resp., in $Y_{\sbullet}$) of degree less than or equal to $D_{1,e_0}mN_m$ (resp., $1$).
Then $w_0\pmod*{I_X}$ vanishes along $W_m$ and defines $W_m$ in $X_{e_0}\times_k\PP^{N_m}$.

By the same method starting from $e_j\in R_1$, we can construct an equation
\[
 w_j=\sum\text{(homogeneous in $X_{\sbullet}$ of degree at most $D_{1,e_j}mN_m$)}\times\text{(linear in $Y_{\sbullet}$)}
\]
that vanishes along $W_m$ and defines $W_m$ in $X_{e_j}\times_k\PP^{N_m}$.
Let $w_{N_1+1},\dots,w_h\in k[X_0,\dots,X_n]$ be homogeneous polynomials that generate $I_X$.
Notice that the bihomogeneous ideal
\begin{equation}
 (w_0,\dots,w_{N_1},w_{N_1+1},\dots,w_h)\subseteq k[X_0,\dots,X_n;Y_0,\dots,Y_m]
\end{equation}
may \emph{not} be prime but the closed subscheme defined by $(w_0,\dots,w_h)$ in $\PP^n\times_k\PP^{N_m}$ coincides with $W_m$.

Set
\begin{equation}
 D_1:=\max_{0\leq i\leq N_1}\{D_{1,e_i}\},\quad D_2:=\max_{N_1+1\leq j\leq h}\{\deg_{X_{\sbullet}}\! w_j\},
\end{equation}
which does not depend on $m$.
By the Euler rule together with the Jacobian criterion in the affine case, we conclude that the singular locus $\Sing(\pi_m)\subseteq X\times_k\PP(R_m^{\vee})$ is defined by the determinants of certain $(n-\dim X+1)\times (n-\dim X+1)$-minors of the Jacobian matrix $\left(\frac{\partial w_i}{\partial X_j}\right)$, whose degrees in $X_{\sbullet}$ (resp., in $Y_{\sbullet}$) are all bounded from above by $(N_1+1)(D_1mN_m-1)+(n-\dim X)(D_2-1)$ (resp., by $N_1+1$).
We choose a positive constant $D'>0$ such that
\[
 (N_1+1)(D_1mN_m-1)+(n-\dim X)(D_2-1)\leq D'm^{\kappa(R_{\sbullet})+1}
\]
for all $m\geq 1$.
Let $\varphi_X(l)$ be as in (\ref{eqn:varphil}) and set
\begin{equation}
 P(m):=\varphi_X(\deg X(D'm^{\kappa(R_{\sbullet})+1})^{\dim X+1})\cdot (N_1+1).
\end{equation}
Then $\deg P=\dim X(\dim X+1)(\kappa(R_{\sbullet})+1)$.
Since $\pi_m(\Sing(\pi_m))$ is properly contained in $\PP(R_m^{\vee})$ due to Kleiman \cite[Corollaries~5 and 12]{Kleiman74}, we can apply Lemma~\ref{lem:degree} to this situation by setting
\[
 p=D'm^{\kappa(R_{\sbullet})+1}\quad\text{and}\quad q=N_1+1.
\]
Then we conclude that there exists a hypersurface $Z_m\subseteq\PP(R_m^{\vee})$ having degree less than or equal to $P(m)$ and containing $\pi_m(\Sing(\pi_m))$.
\end{proof}

By applying Theorem~\ref{thm:degBertini} to the image of $R_m$ via $\Hz(X,mA)\to\Hz(Y,mA|_Y)$, we have the following.

\begin{corollary}\label{cor:Bertini}
Under the same assumptions as in Theorem~\ref{thm:degBertini}, let $Y$ be a smooth closed subvariety of $X$, and let $y_1,\dots,y_q$ be closed points on $X$.
Then one can find a polynomial function $P(m)$ and hypersurfaces $Z_m\subseteq\PP(R_m^{\vee})$ for $m=1,2,\dots$ having the following two properties.
\begin{enumerate}
\item[\textup{(1)}]{$\deg P\leq\dim Y(\dim Y+1)(\kappa(R_{\sbullet})+1)+q$.}
\item[\textup{(2)}]{For every $m\geq 1$, the hypersurface $Z_m\subseteq\PP(R_m^{\vee})$ contains the set
\[
 \left\{H\in\PP(R_m^{\vee})\,\left|\,\begin{array}{l}\text{$\phi_m(Y)\subseteq H$, $\phi_m^{-1}(H)\cap Y$ is not smooth,}\\ \text{or $H$ contains one of $y_1,\dots,y_q$}\end{array}\right.\right\}
\]
and the homogeneous degree of $Z_m$ in $\PP(R_m^{\vee})$ is less than or equal to $P(m)$.}
\end{enumerate}
\end{corollary}

\section{Proofs}

In this section, we turn to the arithmetic case and give proofs of Theorem~A and Corollaries~B and C.
To prove Theorem~A, we use Lemmas~\ref{lem:CN}, \ref{lem:poschr}, and \ref{lem:ZhangMoriwaki}.

\begin{lemma}[\textup{Combinatorial Nullstellensatz \cite[Lemma 5.2]{Moriwaki95}, \cite[Theorem 1.2]{Alon99}}]\label{lem:CN}
Let $V$ be a finite-dimensional vector space over a field $k$, and let
\[
 u:V\to k
\]
be a nonzero polynomial function with maximal total degree $\deg u$.
Let $e_1,\dots,e_N$ be generators of $V$ over $k$, and let $S_1,\dots,S_N$ be subsets of $k$.
If $\Card(S_j)\geq \deg u+1$ for every $j$, then there exist $a_1\in S_1,\dots,a_N\in S_N$ such that
\[
 u(a_1e_1+\dots+a_Ne_N)\neq 0.
\]
\end{lemma}

\begin{lemma}\label{lem:poschr}
Let $X$ be a projective arithmetic variety, let $A$ be a line bundle on $X$, and let $R_{\sbullet}$ be a graded linear series belonging to $A$.
Suppose that $R_1$ is base point free.
Let $y_1,\dots,y_l\in X$ be distinct closed points on $X$ such that $\chr(k(y_i))\neq 0$ for every $i$, and let $e_1^{(m)},\dots,e_{N_m}^{(m)}\in R_m$ be generators of the $\ZZ$-module $R_m$.
Set $F:=\prod_{\substack{\text{$p$: prime} \\ \exists i,\,p|\chr(k(y_i))}}p$.
Then, for every sufficiently large $m$, there exist integers $a_1,\dots,a_{N_m}$ such that $0\leq a_j<F$ for every $j$, and
\[
 (a_1+Fb_1)e_1^{(m)}(y_i)+\dots+(a_{N_m}+Fb_{N_m})e_{N_m}^{(m)}(y_i)\neq 0
\]
for every integer $b_1,\dots,b_{N_m}$ and for every $i$.
\end{lemma}

\begin{proof}
First, we need the following claim.

\begin{claim}
For every sufficiently large $m$, there exists an $s\in R_m$ such that $s(y_i)\neq 0$ for every $i$.
\end{claim}

\begin{proof}
Let $\phi:X\to\PP_{\ZZ}^{N_1}$ be the morphism associated to $R_1$ such that $\phi^*X_j=e_j^{(1)}$ for every $j$, and let $\mathcal{O}(1)$ be the hyperplane line bundle on $\PP_{\ZZ}^{N_1}$.
Then, for every sufficiently large $m$, the homomorphism
\[
 \Hz(\PP_{\ZZ}^{N_1},\mathcal{O}(m))\to\bigoplus_{i}\mathcal{O}(m)(\phi(y_i))
\]
is surjective.
Let $t\in\Hz(\PP_{\ZZ}^{N_1},\mathcal{O}(m))$ be a section such that $t(\phi(y_i))\neq 0$ for every $i$.
Then $s:=\phi^*t$ has the desired property.
\end{proof}

Next, let $s\in R_m$ as above.
Since $Fe_j^{(m)}(y_i)=0$ for every $i,j$, we have that
\[
 (s+Ft)(y_i)=s(y_i)\neq 0
\]
for every $t\in R_m$ and for every $i$.
Thus we conclude the claim.
\end{proof}

\begin{lemma}[\textup{Zhang-Moriwaki \cite[Theorem~A and Corollary~B]{Moriwaki10}}]\label{lem:ZhangMoriwaki}
Under the same assumptions as in Theorem~A, take an $m_0\gg 1$, and fix $e_1,\dots,e_N\in R_{m_0}$ such that
\[
 \{x\in X_{\QQ}\,|\,e_1(x)=\dots=e_N(x)=0\}=\emptyset
\]
and such that $\|e_j\|_{m_0}<1$ for every $j$.
Then there exists a positive constant $C>0$ such that, for every sufficiently large $m$, one can find a $\ZZ$-basis $e_1^{(m)},\dots,e_{N_m}^{(m)}$ for $R_m$ such that
\[
 \max_i\left\{\|e_i^{(m)}\|_m\right\}\leq Cm^{(\dim X+2)(\dim X-1)}\left(\max_j\left\{\|e_j\|_{m_0}\right\}\right)^{m/m_0}.
\]
\end{lemma}

\begin{proof}[Proof of Theorem~A]
Let $r:=[K:\QQ]$, and let $X(\CC)=X_1\cup\dots\cup X_r$ be the decomposition into connected components.
Let $R_{m,\alpha}$ be the image of $R_m\otimes_{\ZZ}\CC$ via $\Hz(X,A)\otimes_{\ZZ}\CC\to\Hz(X_{\alpha},A_{\CC}|_{X_{\alpha}})$, and let $\phi_{m,\alpha}:X_{\alpha}\to\PP_{\CC}^{M_m}$ be a morphism associated to $R_{m,\alpha}$, where we set $M_m:=\rk_{\ZZ} R_m/r$.
By Lemma~\ref{lem:ZhangMoriwaki}, there exist constants $C,Q$ with $C>0$ and $0<Q<1$ such that there exists a $\ZZ$-basis $e_1^{(m)},\dots,e_{rM_m}^{(m)}$ for $R_m$ consisting of the sections with supremum norms less than or equal to
\begin{equation}
 Cm^{(\dim X+2)(\dim X-1)}Q^m.
\end{equation}

For each $Y^j$, there exists a unique component $X_{\alpha(j)}$ that contains $Y^j$.
Suppose that $\chr(x_i)=0$ for $i=1,\dots,q_1$ and $\chr(x_i)\neq 0$ for $i=q_1+1,\dots,q=q_1+q_2$ and let $y_i$ be a closed point in $\overline{\{x_i\}}$.
By applying Corollary~\ref{cor:Bertini} to $X_{\alpha(j)}$, $Y^j$, $y_1,\dots,y_{q_1}$, and $R_{\sbullet,\alpha(j)}$, one can find a polynomial function $P_j(m)$ of degree less than or equal to $\dim Y^j(\dim Y^j-1)(\kappa(R_{\sbullet,\alpha(j)})+1)+q_1$ and hypersurfaces $Z_{m,j}\subseteq\PP(R_{m,\alpha(j)}^{\vee})$ defined by homogeneous polynomials $u_{m,j}$ of degree less than or equal to $P_j(m)$, respectively, such that $Z_{m,j}$ contains all the hyperplanes $H$ in $\PP(R_{m,\alpha(j)}^{\vee})$ such that $\phi_{m,\alpha(j)}(Y^j)\subseteq H$, $\phi_{m,\alpha(j)}^{-1}(H)\cap Y^j$ is not smooth, or $\phi_{m,\alpha(j)}^{-1}(H)$ contains one of $y_1,\dots,y_{q_1}$.
Set
\[
 u_{m,\alpha}:=\prod_{\alpha(j)=\alpha}u_{m,j}
\]
and consider the homogeneous polynomial function
\[
 u:R_m\otimes_{\ZZ}\CC\xrightarrow{\sim}\bigoplus_{\alpha=1}^rR_{m,\alpha}\xrightarrow{\prod_{\alpha}u_{m,\alpha}}\CC
\]
of degree less than or equal to
\begin{equation}
 P(m):=P_1(m)+\dots +P_p(m).
\end{equation}

Set $F:=\prod_{\substack{\text{$q$: prime} \\ \exists i,\,q|\chr(y_i)}}q.$
Since $e_1^{(m)},\dots,e_{rM_m}^{(m)}\in R_m$ generate $R_m\otimes_{\ZZ}\CC$ over $\CC$, one can find integers $a_1,\dots,a_{rM_m}$ and $b_1,\dots,b_{rM_m}$ such that $0\leq a_i< F$ for every $i$, $0\leq b_j\leq P(m)$ for every $j$, and
\[
 u((a_1+Fb_1)e_1^{(m)}+\dots+(a_{rM_m}+Fb_{rM_m})e_{rM_m}^{(m)})\neq 0
\]
by use of Lemmas~\ref{lem:poschr} and \ref{lem:CN}.
Hence, for each $m\gg 1$, there exists a section $t_m\in R_m$ such that $t_m|_{X_{\alpha}}$ is not contained in any of $Z_{m,j}$ and
\[
 \|t_m\|_m\leq CFrm^{(\dim X+2)(\dim X-1)}M_m(1+P(m))Q^m.
\]
Since the right-hand side tends to zero as $m\to\infty$, we conclude the proof.
\end{proof}

Corollary~B is a direct consequence of Theorem~A.

\begin{proof}[Proof of Corollary~C]
We can take $a_0\gg 1$ such that $\Bs\mathrm{F}^{0+}(X,a_0\overline{L})=\SBs^{0+}(\overline{L})$.
Let $\mathfrak{b}^{0+}(a_0\overline{L}):=\Image(\mathrm{F}^{0+}(X,a_0\overline{L})\otimes_{\ZZ}(-a_0L)\to\mathcal{O}_X)$, let $\mu:X'\to X$ be a blowup such that $X'$ is generically smooth and such that $\mu^{-1}\mathfrak{b}^{0+}(a_0\overline{L})\cdot\mathcal{O}_{X'}$ is Cartier, and let $E$ be an effective Cartier divisor on $X'$ such that $\mathcal{O}_{X'}(-E)=\mu^{-1}\mathfrak{b}^{0+}(a_0\overline{L})\cdot\mathcal{O}_{X'}$.
We can assume that $\mu$ is isomorphic over $X\setminus\SBs^{0+}(\overline{L})$ (see \cite{Hironaka64}).
Set $x_i':=\mu^{-1}(x_i)\in X'\setminus E$ for $i=1,\dots,q$.
Let $B:=\mathcal{O}_{X'}(E)$ and let $1_B$ be the canonical section.

\begin{lemma}\label{lem:YuanZhang}
\begin{enumerate}
\item[\textup{(1)}] We can endow $B$ with a continuous Hermitian metric $|\cdot|_{\overline{B}}$ such that
\[
 |1_B|_{\overline{B}}(x)=\max_{\substack{e\in\Hz(X,a_0L) \\ 0<\|e\|_{\sup}^{(a_0)}<1}}\left\{\frac{|e|_{a_0\overline{L}}(\mu(x))}{\|e\|_{\sup}^{(a_0)}}\right\}\leq 1
\]
for all $x\in X'(\CC)$.
\item[\textup{(2)}] We set $\overline{B}:=(B,|\cdot|_{\overline{B}})$ and $\overline{A}:=a_0\mu^*\overline{L}-\overline{B}$.
Then $\overline{A}$ is a continuous Hermitian line bundle on $X'$ such that
\[
 \Bs\mathrm{F}^{0+}(X',\overline{A})=\emptyset\quad\text{and}\quad c_1(\overline{A})\geq 0
\]
as a current.
\end{enumerate}
\end{lemma}

\begin{proof}
Set $\{e\in\Hz(X,a_0L)\setminus\{0\}\,|\,\|e\|_{\sup}^{(a_0)}<1\}=\{e_1,\dots,e_N\}$.

(1): We choose an open covering $\{U_{\nu}\}$ of $X'(\CC)$ such that $a_0\mu^*L_{\CC}|_{U_{\nu}}$ is trivial with local frame $\eta_{\nu}$, and $E_{\CC}\cap U_{\nu}$ is defined by a local equation $g_{\nu}$.
Then we can write $\mu^*e_j=f_{j,\nu}\cdot g_{\nu}\cdot\eta_{\nu}$ on $U_{\nu}$, where $f_{1,\nu},\dots,f_{N,\nu}$ are holomorphic functions on $U_{\nu}$ satisfying $\{x\in U_{\nu}\,|\,f_{1,\nu}(x)=\dots=f_{N,\nu}(x)=0\}=\emptyset$.
Since
\[
 \max_j\left\{\frac{|e_j|_{a_0\overline{L}}(\mu(x))}{\|e_j\|_{\sup}^{(a_0)}}\right\}=\max_j\left\{\frac{|f_{j,\nu}(x)|}{\|e_j\|_{\sup}^{(a_0)}}\right\}\cdot |\eta_{\nu}|_{a_0\mu^*\overline{L}}(x)\cdot |g_{\nu}(x)|
\]
on $x\in U_{\nu}$, we have (1).

(2): For each $x_0\in X'(\CC)$, we take indices $\nu$ and $j_0$ such that $x_0\in U_{\nu}$ and $f_{j_0,\nu}(x_0)\neq 0$.
Let $\varepsilon_j$ be the section of $A$ such that $\mu^*e_j=\varepsilon_j\otimes 1_B$, and set $h_{j,\nu}:=f_{j,\nu}/f_{j_0,\nu}$.
Then
\[
 -\log|\varepsilon_{j_0}|_{\overline{A}}^2(x)=\max_j\left\{\log|h_{j,\nu}(x)|^2-\log\left(\|e_j\|_{\sup}^{(a_0)}\right)^2\right\}
\]
is plurisubharmonic near $x_0$.

We claim that $\|\varepsilon_j\|_{\sup}=\|e_j\|_{\sup}^{(a_0)}$, so that $\varepsilon_j\in\mathrm{F}^{0+}(X',\overline{A})$.
The inequality $\|\varepsilon_j\|_{\sup}\geq\|e_j\|_{\sup}^{(a_0)}$ is clear.
Since
\[
 |\varepsilon_j|_{\overline{A}}(x)=|e_j|_{a_0\overline{L}}(\mu(x))\cdot\min_i\left\{\frac{\|e_i\|_{\sup}^{(a_0)}}{|e_i|_{a_0\overline{L}}(\mu(x))}\right\}\leq \|e_j\|_{\sup}^{(a_0)}
\]
for all $x\in (X'\setminus E)(\CC)$, we have $\|\varepsilon_j\|_{\sup}=\|e_j\|_{\sup}^{(a_0)}$.
This means that $\Bs\mathrm{F}^{0+}(X',\overline{A})=\emptyset$.
\end{proof}

We apply Corollary~B to $\overline{A}$ and we can find an $m\gg 1$ and a $\sigma\in\Hz(X',mA)$ such that $\div(\sigma)_{\QQ}$ is smooth, $\sigma(x_i')\neq 0$ for every $i$, and $\|\sigma\|_{\sup}<1$.
Since $X$ is normal, there exists an $s\in\Hz(X,ma_0L)$ such that $\mu^*s=\sigma\otimes 1_B^{\otimes m}$.
Since $\mu$ is isomorphic over $X\setminus\SBs^{0+}(\overline{L})$, $s$ has the desired properties.
\end{proof}

\bibliography{ikoma}
\bibliographystyle{plain}

\end{document}